\documentclass[reqno,11pt]{amsart}

\numberwithin{equation}{section}

\usepackage{verbatim} 
\usepackage{xspace} 
\usepackage{amssymb}
\usepackage{amsmath}
\usepackage{amsthm} %
\usepackage[latin1]{inputenc}
\usepackage{graphicx}
\usepackage{amscd,amssymb,euscript,graphics}
\makeindex

\newcommand{\R}{\ensuremath{\mathbb{R}}}
\newcommand{\C}{\ensuremath{\mathbb{C}}}

\newcommand{\norm}[1]{\ensuremath{\|#1\|}}

\newtheorem{teo}{Theorem}[section]

\newtheorem{lem}[teo]{Lemma}
\newtheorem{cor}[teo]{Corollary}

\newtheorem{ej}[teo]{Example}
\newtheorem{obs}[teo]{Remark}

\newcount\theTime
\newcount\theHour
\newcount\theMinute
\newcount\theMinuteTens
\newcount\theScratch
\theTime=\number\time
\theHour=\theTime
\divide\theHour by 60
\theScratch=\theHour
\multiply\theScratch by 60
\theMinute=\theTime
\advance\theMinute by -\theScratch
\theMinuteTens=\theMinute
\divide\theMinuteTens by 10
\theScratch=\theMinuteTens
\multiply\theScratch by 10
\advance\theMinute by -\theScratch

\def\today{{\number\day\space
 \ifcase\month\or
  January\or February\or March\or April\or May\or June\or
  July\or August\or September\or October\or November\or December\fi
 \space\number\year}}

\begin{document}

\author{Gabriel H. Tucci}
\title[Limits laws for geometric means of free random variables]{Limits laws for geometric means of free random variables}
\address{Bell Laboratories, 600 Mountain Avenue, Murray Hill, NJ 07974}
\email{gabriel.tucci@.alcatel-lucent.com}


\begin{abstract}
Let $\{T_{k}\}_{k=1}^{\infty}$ be a family of $*$--free identically distributed operators in a finite von Neumann algebra. In this work we prove a multiplicative version of the free central limit Theorem. More precisely, let $B_{n}=T_{1}^{*}T_{2}^{*}\ldots T_{n}^{*}T_{n}\ldots T_{2}T_{1}$ then $B_{n}$ is a positive operator and $B_{n}^{1/2n}$ converges in distribution to an operator $\Lambda$. We completely determine the probability distribution $\nu$ of $\Lambda$ from the distribution $\mu$ of $|T|^{2}$. This gives us a natural map $\mathcal{G}:\mathcal{M_{+}}\to \mathcal{M_{+}}$ with $\mu\mapsto \mathcal{G}(\mu)=\nu.$ We study how this map behaves with respect to additive and multiplicative free convolution. As an interesting consequence of our results, we illustrate the relation between the probability distribution $\nu$ and the distribution of the Lyapunov exponents for the sequence $\{T_{k}\}_{k=1}^{\infty}$ introduced in \cite{LyaV}.
\end{abstract}

\maketitle

\vspace{0.3cm}

\section{Introduction}

\vspace{0.2cm}
\noindent Denote by $\mathcal{M}$ the family of all compactly supported probability measures defined in the real line $\R$. We denote by $\mathcal{M}_{+}$ the set of all measures in $\mathcal{M}$ which are supported on $[0,\infty)$. On the set $\mathcal{M}$ there are defined two associative composition laws denoted by $*$ and $\boxplus$. The measure $\mu*\nu$ is the classical convolution of $\mu$ and $\nu$. In probabilistic terms, $\mu*\nu$ is the probability distribution of $X+Y$, where $X$ and $Y$ are commuting independent random variables with distributions $\mu$ and $\nu$, respectively. The measure $\mu\boxplus\nu$ is the free additive convolution of $\mu$ and $\nu$ introduced by Voiculescu \cite{Voi}. Thus, $\mu\boxplus\nu$ is the probability distribution of $X+Y$, where $X$ and $Y$ are free random variables with distribution $\mu$ and $\nu$, respectively. There is a free analogue of multiplicative convolution also. More precisely, if $\mu$ and $\nu$ are measures in $\mathcal{M}_{+}$ we can define $\mu\boxtimes\nu$ the multiplicative free convolution by the probability distribution of $X^{1/2}YX^{1/2}$, where $X$ and $Y$ are free random variables with distribution $\mu$ and $\nu$, respectively.

\vspace{0.3cm}
\noindent In this paper we prove a multiplicative version of the free central limit Theorem. More precisely, let $\{T_{k}\}_{k=1}^{\infty}$ be a family of $*$--free identically distributed operators in a finite von Neumann algebra. Let $B_{n}$ the positive operator defined as 
$$
B_{n}:=T_{1}^{*}T_{2}^{*}\ldots T_{n}^{*}T_{n}\ldots T_{2}T_{1}.
$$ 
Then  $B_{n}^{\frac{1}{2n}}$ converges in distribution to a positive operator $\Lambda$. We completely determine the probability distribution $\nu$ of $\Lambda$ from the probability distribution of $|T|^{2}$. Our first observation is that it is enough to restrict ourselves to positive operators. In other words, let $a_{k}=|T_{k}|$ then 
$$
B_{n}=T_{1}^{*}T_{2}^{*}\ldots T_{n}^{*}T_{n}\ldots T_{2}T_{1}
$$ 
has the same distribution as 
$$
b_{n}=a_{1}a_{2}\ldots a_{n}^{2}\ldots a_{2}a_{1}
$$ 
for all $n\geq 1$. Hence, to prove that $B_{n}^{\frac{1}{2n}}$ converges in distribution it is enough to prove that $b_{n}^{\frac{1}{2n}}$ converges in distribution. 

\vspace{0.3cm}
\noindent Our main result is the following, let $\mu$ be the probability distribution of $|T_{k}|^{2}$ then 
$$
B_{n}^{\frac{1}{2n}}\longrightarrow \Lambda\quad\text{in distribution}.
$$
Let $\nu$ be the probability distribution of $\Lambda$, then 
\begin{equation}\label{eqf}
\nu=\beta\delta_{0}+\sigma \quad\quad\text{with}\quad\quad d\sigma=f(t)\,\mathbf{1}_{\big(\norm{|T_{1}|^{-1}}_2^{-1},\norm{|T_{1}|}_2\big]}(t)\,dt
\end{equation}
where $\beta=\mu(\{0\})$, $f(t)=\big(F^{<-1>}_{\mu}\big)'(t)$ and $F_{\mu}(t)=S_{\mu}(t-1)^{-1/2}$ ($F^{<-1>}_{\mu}$ is the inverse with respect to composition of $F_{\mu}$). 

\vspace{0.3cm}
\noindent This gives us, naturally, a map 
$$
\mathcal{G}:\mathcal{M_{+}}\to \mathcal{M_{+}}\quad\text{with}\quad \mu\mapsto \mathcal{G}(\mu)=\nu.
$$ 
The measure $\mathcal{G}(\mu)$ is a compactly supported positive measure with at most one atom at zero and $\mathcal{G}(\mu)(\{0\})=\mu(\{0\})$. 

\vspace{0.3cm}
\noindent We would like to mention that Vladislav Kargin in Theorem 1 of \cite{Vla} proved an estimate in the norm of the positive operators $b_{n}$. More precisely, he proved that if $\tau(a_{1}^{2})=1$ there exists a positive constant $K>0$ such that
$$
\sqrt{n}\,\sigma(a_{1}^{2})\leq \norm{b_{n}}\leq K\,n\,\norm{a_{1}^{2}}
$$
where $\sigma^{2}(x)=\tau(x^{2})-\tau(x)^{2}$.

\vspace{0.3cm}
\noindent  It is interesting to compare this result with the analogous result in the classical case. Let $\{a_{k}\}_{k=1}^{\infty}$ be independent positive identically distributed commutative random variables with distribution $\mu$. Applying the Law of the Large Numbers to the random variables $\log(a_{k})$, in case $\log(a_{k})$ is integrable, or applying Theorem 5.4 in \cite{DykHa} in the general case, we obtain that
$$
\Big(a_{1}a_{2}\ldots a_{n}\Big)^{1/n}\longrightarrow \,e^{\,\tau\,(\,\log(\,a_{1})\,)}\in [0,\infty)
$$
where the convergence is pointwise.

\vspace{0.3cm}
\noindent  The Lyapunov exponents of a sequence of random matrices was investigated in the pioneering paper of Furstenberg and Kesten \cite{FusKes} and by Oseledec in \cite{Os}. Ruelle \cite{Ru} developed the theory of Lyapunov exponents for random compact linear operators acting on a Hilbert space. Newman in \cite{N1} and \cite{N2} and later Isopi and Newman in \cite{IsN} studied Lyapunov exponents for random $N\times N$ matrices as $N\to\infty$. Later on, Vladislav Kargin \cite{LyaV} investigated how the concept of Lyapunov exponents can be extended to free linear operators (see \cite{LyaV} for a more detailed exposition).

\vspace{0.3cm}
\noindent  In our case, given $\{a_{k}\}_{k=1}^{\infty}$ be free positive identically distributed random variables. Let $\mu$ be the spectral probability distribution of $a_{k}^2$ and assume that $\mu(\{0\})=0$. Then 
$$
\Big(a_{1}a_{2}\ldots a_{n}^{2}\ldots a_{2}a_{1}\Big)^{\frac{1}{2n}}\longrightarrow \Lambda
$$ 
where $\Lambda$ is a positive operator. The probability distribution of the Lyapunov exponents associated to the sequence $\{a_{k}\}_{k=1}^{\infty}$, is the spectral probability distribution $\gamma$ of the selfadjoint operator $L:=\ln(\Lambda)$. Moreover, $\gamma$ is absolutely continuous with respect to Lebesgue measure and has Radon--Nikodym derivative given by
$$
d\gamma(t)=e^{t}f(e^{t})\,\mathbf{1}_{\big(\ln\norm{a_{1}^{-1}}_{2}^{-1},\,\ln\norm{a_1}_2\big]}(t)\,dt
$$
where the function $f(t)$ is as in equation (\ref{eqf}).

\vspace{0.3cm}
\noindent Now we will describe the content of this paper. In Section $\S2$, we recall some preliminaries as well as some known results and fix the notation. In Section $\S3$, we prove our main Theorem and study how the map $\mathcal{G}$ behaves with respect to additive and multiplicative free convolution.  In Section $\S4$, we present some examples. Finally, in Section $\S5$, we derive the probability distribution of the Lyapunov exponents of the sequence $\{a_{k}\}_{k=1}^{\infty}$.

\vspace{0.4cm}
\noindent{\it Acknowledgment:} I thank my advisor, Ken Dykema, for many helpful discussions and comments.

\section{Preliminaries and Notation}
\vspace{0.2cm}
\noindent We begin with an analytic method for the calculation of multiplicative free convolution discovered by Voiculescu. Denote $\C$ the complex plane and set $\C^{+}=\{z\in\C\,\,:\,\,\mathrm{Im}(z)>0\}$, $\C^{-}=-\C^{+}$. For a measure $\nu\in\mathcal{M}_{+}\setminus\{\delta_{0}\}$ one defines the analytic function $\psi_{\nu}$ by 
$$
\psi_{\nu}(z)=\int_{0}^{\infty}{\frac{zt}{1-zt}\,d\nu(t)}
$$
for $z\in\C\setminus{[0,\infty)}$. The measure $\nu$ is completely determined by $\psi_{\nu}$. The function $\psi_{\nu}$ is univalent in the half-plane $i\C^{+}$, and $\psi_{\nu}(i\C^{+})$ is a region contained in the circle with center at $-1/2$ and radius $1/2$. Moreover, $\psi_{\nu}(i\C^{+})\cap (-\infty,0]=(\beta-1,0)$, where $\beta=\nu(\{0\})$. If we set $\Omega_{\nu}=\psi_{\nu}(i\C^{+})$, the function $\psi_{\nu}$ has an inverse with respect to composition
$$
\chi_{\nu}:\Omega_{\nu}\to i\C^{+}.
$$
\noindent Finally, define the $S$--transform of $\nu$ to be 
$$
S_{\nu}(z)=\frac{1+z}{z}\chi_{\nu}(z)\,\,,\quad\quad z\in\Omega_{\nu}.
$$
See \cite{BerPata} for a more detailed exposition. The following is a classical Theorem originally proved by Voiculescu and generalized by Bercovici and Voiculescu in \cite{BerVoi} for measures with unbounded support.

\vspace{0.3cm}
\begin{teo}
Let $\mu,\nu\in\mathcal{M}_{+}$. Then 
$$
S_{\mu\boxtimes\nu}(z)=S_{\mu}(z)S_{\nu}(z)
$$
\noindent for every $z$ in the connected component of the common domain of $S_{\mu}$ and $S_{\nu}$.  
\end{teo}

\vspace{0.3cm}
\noindent  It was shown by Hari Bercovici and Dan Voiculescu in \cite{Ber} that the additive free convolution of probability measures on
the real line tend to have a lot fewer atoms. More precisely we have the following Theorem.
\vspace{0.2cm}
\begin{teo}\label{Berco}
Let $\mu$ and $\nu$ be two probability measures supported in $\R$. The number $a$ is an atom
for the free additive convolution of $\mu$ and $\nu$ if and only if a can be written as
$a = b + c$ where $\mu(\{b\})+\nu(\{c\})>1$. In this case, $\mu\boxplus\nu\,(\{a\})=\mu(\{b\})+\nu(\{c\})-1$.
\end{teo}

\vspace{0.3cm}
\noindent  For measures supported on the positive half-line, an analogous result holds, with
a difference when zero is an atom. The following Theorem was proved by Serban Belinschi in \cite{B}.
\vspace{0.2cm}
\begin{teo}\label{SB}
 Let $\mu$ and  $\nu$ be two probability measures supported in $[0,\infty)$. 
\begin{enumerate}
 \item The following are equivalent

\vspace{0.2cm}

\begin{enumerate}

\item $\mu\boxtimes \nu$ has an atom at $a>0$

\vspace{0.2cm} 	

\item there exists $u$ and $v$ so that $uv=a$ and $\mu(\{u\})+\nu(\{v\})>1$.\\ Moreover, $\mu(\{u\})+\nu(\{v\})-1=\mu\boxtimes \nu\,(\{a\})$.

\end{enumerate}

\vspace{0.2cm} 
\item $\mu\boxtimes \nu\,(\{0\})=\max\{\mu(\{0\}),\nu(\{0\})\}$.
\end{enumerate}

\end{teo}

\vspace{0.2cm}
\noindent  In \cite{NS} Nica and Speicher introduced the class of $R$--diagonal operators in a non commutative C$^{*}$-probability space. An operator $T$ is $R$--diagonal if $T$ has the same $*$--distribution as a product $UH$ where $U$ and $H$ are $*$--free, $U$ is a Haar unitary, and $H$ is positive.  
\vspace{0.2cm}
\noindent The next Theorem and Corollary were proved by Uffe Haagerup and Flemming Larsen (\cite{HaF}, Theorem 4.4 and the Corollary following it) where they completely characterized the Brown measure of an $R$--diagonal element.
\vspace{0.3cm}
\begin{teo}\label{Mteo}
Let $(M,\tau)$ be a non--commutative finite von Neumann algebra with a faithful trace $\tau$. Let $u$ and $h$ be $*$--free random variables in $M$, $u$ a Haar unitary, $h\geq 0$ and assume that the distribution $\mu_{h}$ for $h$ is not a Dirac measure. Denote $\mu_{T}$ the Brown measure for $T=uh$. Then
\begin{enumerate}
\item $\mu_{T}$ is rotation invariant and 
$$
\mathrm{supp}(\mu_{T})=[\norm{h^{-1}}_{2}^{-1},\norm{h}_2]\times_{p}[0,2\pi).
$$

\item The $S$ transform $S_{h^{2}}$ of $h^{2}$ has an analytic continuation to neighborhood of the interval $(\mu_{h}(\{0\})-1,0]$, $S_{h^{2}}((\mu_{h}(\{0\})-1,0])=[\norm{h}_{2}^{-2},\norm{h^{-1}}_{2}^{2})$ and $S_{h^{2}}^{'}<0$ on $(\mu_{h}(\{0\})-1,0)$.

\vspace{0.3cm}

\item $\mu_{T}(\{0\})=\mu_{h}(\{0\})$ and $\mu_{T}(B(0,S_{h^2}(t-1)^{-1/2})=t$ for $t\in (\mu_{h}(\{0\}),1]$.

\vspace{0.3cm}

\item $\mu_{T}$ is the only rotation symmetric probability measure satisfying (3).
\end{enumerate}
\end{teo}

\vspace{0.3cm}

\begin{cor}\label{Mcor}
With the notation as in the last Theorem we have 
\begin{enumerate}
\item the function $F(t)=S_{h^2}(t-1)^{-1/2}:(\mu_{h}(\{0\}),1]\to (\norm{h^{-1}}_{2}^{-1},\norm{h}_2]$ has an analytic continuation to a neighboorhood of its domain and $F^{\,'}>0$ on $(\mu_{h}(\{0\}),1)$.

\vspace{0.3cm}
\item $\mu_{T}$ has a radial density function $f$ on $(0,\infty)$ defined by 
$$g(s)=\frac{1}{2\pi s}(F^{<-1>})^{'}(s)\,\mathbf{1}_{(F(\mu_{h}(\{0\})),F(1)]}(s).$$

\end{enumerate}

\vspace{0.2cm}
\noindent Therefore, $\mu_{T}=\mu_{h}(\{0\})\delta_{0}+\sigma$ with $d\sigma=g(|\lambda|)dm_{2}(\lambda)$.
\end{cor}

\section{Main Results}

\vspace{0.2cm}
\noindent In this Section we prove our main results. Let us first fix some notation. We say two operators $A$ and $B$ in a finite von Neumann algebra $(\mathcal{N},\tau)$ have the same $*$--distribution iff $\tau(p(A,A^{*}))=\tau(p(B,B^{*}))$ for all non--commutative polynomials $p\in\C\langle X,Y \rangle$. In this case we denote $A\sim_{*d} B$. If $A$ and $B$ are self--adjoint we say that $A$ and $B$ have the same distribution and we denote it by $A\sim_{d} B$.

\vspace{0.3cm}
\begin{lem}
Let $\{T_{k}\}_{k=1}^{\infty}$ be a family of $*$--free identically distributed operators in a finite von Neumann algebra. Let $a_{k}=|T_{k}|$ be the modulus of $T_{k}$. Then the positive operators $B_{n}=T_{1}^{*}T_{2}^{*}\ldots T_{n}^{*}T_{n}\ldots T_{2}T_{1}$ and $b_{n}=a_{1}a_{2}\ldots a_{n}^{2}\ldots a_{2}a_{1}$ have the same distribution.
\end{lem}

\begin{proof}
Let $T_{k}=u_{k}a_{k}$ be the polar decomposition of the operator $T_{k}$. Since we are in a finite von Neumann algebra we can always extend $u_{k}$ to be a unitary (see \cite{Tak}). We will proceed by induction on $n$. The case $n=1$ is obvious since $T_{1}^{*}T_{1}=a_{1}^{2}$. Assume now that $B_{k}$ has the same distribution as $b_{k}$ for $k<n$. Then by $*$--freeness and the induction hypothesis
$$
B_{n}=T_{1}^{*}T_{2}^{*}\ldots T_{n}^{*}T_{n}\ldots T_{2}T_{1}\sim_{d} (u_{1}a_{1})^{*}(a_{2}\ldots a_{n}^{2}\ldots a_{2})(u_{1}a_{1}).
$$
Hence
$$
B_{n}\sim_{d}a_{1}u_{1}^{*}(a_{2}\ldots a_{n}^{2}\ldots a_{2})u_{1}a_{1}=u_{1}^{*}(u_{1}a_{1}u_{1}^{*})(a_{2}\ldots a_{n}^{2}\ldots a_{2})(u_{1}a_{1}u_{1}^{*})u_{1}.$$
Since conjugating by a unitary does not alter the distribution we see that
$$
B_{n}\sim_{d} (u_{1}a_{1}u_{1}^{*})(a_{2}\ldots a_{n}^{2}\ldots a_{2})(u_{1}a_{1}u_{1}^{*}).
$$
Since the operators $\{T_{k}\}_{k=1}^{\infty}$ are $*$--free then $\{\{u_{k},a_{k}\}\}_{k}^{\infty}$ is a $*$--free family and $a_{1}\sim_{d}{u_{1}a_{1}u_{1}^{*}}$ and are free with respect to $\{a_{k}\}_{k\geq 2}$. Then, by freeness, 
$$
B_{n}\sim_{d}(u_{1}a_{1}u_{1}^{*})(a_{2}\ldots a_{n}^{2}\ldots a_{2})(u_{1}a_{1}u_{1}^{*})\sim_{d} a_{1}a_{2}\ldots a_{n}^{2}\ldots a_{2}a_{1}
$$
concluding the proof.
\end{proof}

\noindent Now we are ready to prove our main Theorem.

\begin{teo}\label{main}
Let $\{T_{k}\}_{k}$ be a sequence of $*$--free equally distributed operators. Let $\mu$ in $\mathcal{M}_{+}$ be the distribution of $|T_{k}|^{2}$ and let $B_{n}$ be as in the previous Lemma. The sequence of positive operators $B_{n}^{\frac{1}{2n}}$ converges in distribution to a positive operator $\Lambda$ with distribution $\nu$ in $\mathcal{M}_{+}$. Moreover,
$$
\nu=\beta\delta_{0}+\sigma\quad\quad\text{with}\quad\quad d\sigma=f(t)\,\mathbf{1}_{\big(\norm{|T_{1}|^{-1}}_2^{-1},\norm{|T_{1}|}_2\big]}(t)\,dt
$$
where $\beta=\mu(\{0\})$, $f(t)=\big(F^{<-1>}_{\mu}\big)'(t)$ and $F_{\mu}(t)=S_{\mu}(t-1)^{-1/2}$.
\end{teo}

\begin{proof}
From the previous Lemma it is enough to prove the Theorem for $a_{k}=|T_{k}|$. Let $u$ a Haar unitary $*$--free with respect to the family $\{a_{k}\}_{k}$ and let $h=a_{1}$. Let $T$ be the $R$--diagonal operator defined by $T=uh$.  Given $u$ a Haar unitary and $h$ a positive operator $*$--free from $h$ it is known (see \cite{Free1}, \cite{Free2}) that the family of operators $\{u^{k}h(u^{*})^k\}_{k=0}^{\infty}$ is free. Therefore, defining by $c_{k}=u^{k}h(u^{*})^{k}$ we see that $T^{*}T\sim_{d}c_{1}^2$, $(T^{*})^{2}T^2\sim_d c_{2}c_{1}^2c_{2}$ and it can be shown by induction that
$$
(T^{*})^{n}T^{n}\sim_{d} c_{n}c_{n-1}\cdots c_{1}^2 \cdots c_{n-1}c_{n}.
$$
Therefore, since $c_{k}$ has the same distribution than $a_{k}$, and both families are free, we conclude that the operators $(T^{*})^{n}T^{n}$ and $b_{n}$ have the same distribution. Moreover, by Theorem 2.2 in \cite{HaS} the sequence $\big[(T^{*})^{n}T^{n}\big]^{\frac{1}{2n}}$ converges in distribution to a positive operator $\Lambda$. Let $\nu$ be the probability measure distribution of $\Lambda$. If the distribution of $a_{k}^{2}$ is a Dirac delta, $\mu=\delta_{\lambda}$, then $h=\sqrt{\lambda}$ and 
$$
\big[(T^{*})^{n}T^{n}\big]^{\frac{1}{2n}}=\big[\lambda^{n}(u^{*})^{n}u^{n}\big]^{\frac{1}{2n}}=\sqrt{\lambda}.
$$ 
\noindent Therefore, $b_{n}^{\frac{1}{2n}}$ has the Dirac delta distribution distribution $\delta_{\sqrt{\lambda}}$ and $\nu=\delta_{\sqrt{\lambda}}$. If the distribution of $a_{k}$ is not a Dirac delta, let $\mu_{T}$ the Brown measure of the operator $T$. By Theorem 2.5 in \cite{HaS} we know that
\begin{equation}\label{eqS}
\int_{\C}{|\lambda|^{p}d\mu_{T}(\lambda)}=\lim_{n}{\norm{T^{n}}_{\frac{p}{n}}^{\frac{p}{n}}}
=\lim_{n}{\tau \Big( [(T^{*})^{n}T^{n}]^{\frac{p}{2n}}\Big)}=\tau(\Lambda^{p})=\int_{0}^{\infty}{t^p\,d\nu(t)}.
\end{equation}
\noindent We know by Theorem \ref{Mteo} and Corollary \ref{Mcor} that 
\begin{equation}
\mu_{T}=\beta\delta_{0}+\rho\quad\text{with}\quad d\rho(r,\theta)=\frac{1}{2\pi}f(r)\,\mathbf{1}_{(F_{\mu}(\beta),F_{\mu}(1)]}(r) \,dr d\theta
\end{equation}
where $f(t)=\big(F^{<-1>}_{\mu}\big)'(t)$ and $F_{\mu}(t)=S_{\mu}(t-1)^{-1/2}$. Hence, using equation (\ref{eqS}) we see that
\begin{equation*}
\int_{0}^{\infty}{r^{p}\,d\nu(r)}=\int_{0}^{2\pi}\int_{F_{\mu}(\beta)}^{F_{\mu}(1)}{\frac{1}{2\pi}r^{p}f(r)\,dr d\theta}=\int_{F_{\mu}(\beta)}^{F_{\mu}(1)}{r^{p}f(r)dr}
\end{equation*}
for all $p\geq 1$. Using the fact that if two compactly supported probability measures in $\mathcal{M}_{+}$ have the same  moments then they are equal, we see that
$$
\nu=\beta\delta_{0}+\sigma\quad\text{with}\quad d\sigma=f(t)\,\mathbf{1}_{(F_{\mu}(\beta),F_{\mu}(1)]}(t)\,dt.
$$
\noindent By Corollary \ref{Mcor}, we know that 
$$
F_{\mu}(1)=\norm{a_{1}}_2 \quad\text{and}\quad \lim_{t\to \beta^{+}}{F_{\mu}(t)}=\norm{a_{1}^{-1}}_{2}^{-1}
$$ 
\noindent concluding the proof.
\end{proof}

\noindent Note that the last Theorem gives us a map $\mathcal{G}:\mathcal{M_{+}}\to \mathcal{M_{+}}$ with $\mu\mapsto \mathcal{G}(\mu)=\nu$. The measure $\mathcal{G}(\mu)$ is a compactly supported positive measure with at most one atom at zero and $\mathcal{G}(\mu)(\{0\})=\mu(\{0\})$.\\ 

\noindent Since
$$
\mathcal{G}(\mu)=\beta\delta_{0}+\sigma\quad\quad\text{with}\quad\quad d\sigma=f(t)\,\mathbf{1}_{(F_{\mu}(\beta),F_{\mu}(1)]}(t)\,dt
$$
\noindent and $f(t)=\big(F^{<-1>}_{\mu}\big)'(t)$ where $F_{\mu}(t)=S_{\mu}(t-1)^{\,-1/2}$ for $t\in(\beta,1]$. The function $S_{\mu}(t-1)$ for $t\in(\beta,1]$ is analytic and completely determined by $\mu$. If $\mu_{1}, \mu_{2}\in\mathcal{M}_{+}$ and $S_{\mu_1}(t-1)=S_{\mu_{2}}(t-1)$ in some open interval $(a,b)\subseteq (0,1]$ implies that $\mu_{1}=\mu_{2}$. Therefore, the map $\mathcal{G}$ is an injection.

\vspace{0.3cm}
\begin{obs}A measure $\mu$ in $\mathcal{M_{+}}$ is said $\boxtimes$-infinitely divisible if for each $n\geq 1$ there exists a measure $\mu_{n}$ in  $\mathcal{M_{+}}$ such that
$$
\mu=\mu_{n}\boxtimes\mu_{n}\ldots\boxtimes\mu_{n}\quad (\text{$n$ times}).
$$
\noindent We would like to observe that the image of the map $\mathcal{G}$ is not contained in the set of $\boxtimes$-infinitely divisible laws since an $\boxtimes$-infinitely divisible law cannot have an atom at zero (see Lemma 6.10 in \cite{BerVoi}).
\end{obs}

\vspace{0.2cm}
\noindent The next Theorem investigates how the map $\mathcal{G}$ behaves with respect to additive and multiplicative free convolution. 
\vspace{0.3cm}
\begin{teo}
Let $\mu$ be a measure in $\mathcal{M}_{+}$ and $n\geq 1$. If \,$\mathcal{G}(\mu)=\beta\delta_{0}+\sigma$ with $d\sigma=f(t)\,\mathbf{1}_{(F_{\mu}(\beta),F_{\mu}(1)]}(t)\,dt$ then
$$
\mathcal{G}(\mu^{\boxplus\, n})=\beta_{n}\delta_{0}+\sigma_{n}\quad\text{with}\quad d\sigma_{n}=\sqrt{n}f(t/\sqrt{n})\,\mathbf{1}_{(\sqrt{n}F_{\mu}(\frac{\beta_{n}+n-1}{n})\,,
\sqrt{n}F_{\mu}(1)]}(t)\,dt
$$ 
\noindent where $\beta_{n}=\max\{0,n\beta-(n-1)\}$ and 
$$
\mathcal{G}(\mu^{\boxtimes n})=\beta\delta_{0}+\rho_{n}\quad\text{with}\quad d\rho_{n}=\frac{1}{n}t^{\frac{1-n}{n}}f(t^{1/n})\,\mathbf{1}_{(F_{\mu}(\beta)^{n}\,,F_{\mu}(1)^{n}]}(t)\, dt.
$$
\end{teo}

\begin{proof}

\noindent Recall the relation between the $R_{\mu}$ and $S_{\mu}$ transform (see \cite{HaF}), 
$$
\Big(zR_{\mu}(z)\Big)^{<-1>}=zS_{\mu}(z).
$$ 
\noindent By the fundamental property of the $R$--transform we have $R_{\mu^{\boxplus n}}(z)=nR_{\mu}(z)$. Therefore,
$$
\Big(znR_{\mu}(z)\Big)^{<-1>}=zS_{\mu^{\boxplus n}}(z).
$$
\noindent Hence
$$
\frac{z}{n}S_{\mu}(z/n)=zS_{\mu^{\boxplus n}}(z)
$$
\noindent thus
\begin{equation}
S_{\mu^{\boxplus n}}(z)=\frac{1}{n}S_{\mu}(z/n).
\end{equation}
\noindent Then
$$
F_{\mu^{\boxplus n}}(t)=S_{\mu^{\boxplus n}}(t-1)^{-1/2}=\Bigg(\frac{1}{n}S_{\mu}\Big(\frac{t-1}{n}\Big) \Bigg)^{-1/2}=\sqrt{n}F_{\mu}\Big(\frac{t+n-1}{n}\Big)
$$
it is a direct computation to see that
\begin{equation}
 F_{\mu^{\boxplus n}}^{<-1>}(t)=nF_{\mu}^{<-1>}(t/\sqrt{n})-n+1.
\end{equation}
\noindent By iterating Theorem \ref{Berco} we see that $\mu^{\boxplus\,n}(\{0\})=\max\{0,n\beta-(n-1)\}=\beta_{n}$.

\vspace{0.3cm}
\noindent Now using Theorem \ref{main} we obtain 
$$
\mathcal{G}(\mu^{\boxplus n})=\beta_{n}\delta_{0}+\sigma_{n}\quad\text{with}\quad d\sigma_{n}=\sqrt{n}f(t/\sqrt{n})\,\mathbf{1}_{(\sqrt{n}F_{\mu}(\frac{\beta_{n}+n-1}{n})\,,
\sqrt{n}F_{\mu}(1)\,]}(t)\,dt.
$$
Now let us prove the multiplicative free convolution part, let $\mu^{\boxtimes n}$ then 
$$
S_{\mu^{\boxtimes n}}(z)=S_{\mu}^{n}(z).
$$ 
\noindent Then $F_{\mu^{\boxtimes n}}(t)=F_{\mu}^{n}(t)$ and therefore,
\begin{equation}
 F_{\mu^{\boxtimes n}}^{<-1>}(t)=F_{\mu}^{<-1>}(t^{1/n}).
\end{equation}
\noindent By Theorem \ref{SB} we now that $\mu^{\boxtimes\, n}(\{0\})=\mu(\{0\})=\beta$. Therefore, using Theorem \ref{main} again we obtain 
$$
\mathcal{G}(\mu^{\boxtimes n})=\beta\delta_{0}+\rho_{n}\quad\text{with}\quad d\rho_{n}=\frac{1}{n}t^{\frac{1-n}{n}}f(t^{1/n})\,\mathbf{1}_{(F_{\mu}(\beta)^{n},\,F_{\mu}(1)^{n}]}(t)\, dt.
$$
\end{proof}

\section{Examples}

\noindent In this Section we present some examples of the image of the map $\mathcal{G}$.

\vspace{0.3cm}
\begin{ej}$\mathrm{(Projection)}$
Let $p$ be a projection with $\tau(p)=\alpha$. Then the spectral probability measure of $p$ is $\mu_{p}=(1-\alpha)\delta_{0}+\alpha\delta_{1}$. We would like to compute $\mathcal{G}(\mu_{p})$. Recall that 
$$
S_{p}(z)=\frac{z+1}{z+\alpha}.
$$ 
Therefore,
$$
F_{\mu}(t)=\Big(\frac{t-1+\alpha}{t}\Big)^{1/2} \quad\text{and}\quad F_{\mu}^{<-1>}(t)=\frac{1-\alpha}{1-t^{2}}.
$$
\noindent Hence,
$$
\mathcal{G}(\mu_{p})=(1-\alpha)\delta_{0}+\sigma\quad\text{with}\quad d\sigma=\frac{2t(1-\alpha)}{(t^{2}-1)^{2}}\,\mathbf{1}_{(0,\sqrt{\alpha}]}(t)\,dt.
$$
\end{ej}

\begin{ej}
Let $h$ be a quarter--circular distributed positive operator,
$$
d\mu_{h}=\frac{1}{\pi}\sqrt{4-t^2}\,\mathbf{1}_{[0,2]}(t)\,dt.
$$
\noindent A simple computation shows that 
$$
S_{h^{2}}(z)=\frac{1}{z+1}
$$
\noindent hence by Theorem \ref{main} we see that
$$
d\mathcal{G}(\mu_{h^2})=2t\,\mathbf{1}_{[0,1]}(t)\,dt.
$$ 
\end{ej}

\vspace{0.3cm}
\begin{ej}\label{ejmp}$\mathrm{(Marchenko-Pastur\,\, distribution)}$\\ 

\noindent Let $c>0$ and let $\mu_{c}$ be the Marchenko Pastur or Free Poisson distribution given by
$$
d\mu_{c}=\max\{1-c,0\}\delta_{0}+\frac{\sqrt{(t-a)(b-t)}}{2\pi t}\,\mathbf{1}_{(a,b)}(t)\,dt
$$
where $a=\big(\sqrt{c}-1\big)^2$ and $b=\big(\sqrt{c}+1\big)^2$.\\

\noindent It can be shown (see for example \cite{HaF}) that
$$
S_{\mu_{c}}(z)=\frac{1}{z+c}.
$$
Therefore,
$$
F_{\mu_{c}}(t)=\sqrt{t-1+c}\quad\text{and}\quad F_{\mu_{c}}^{<-1>}(t)=t^{2}+1-c.
$$
Hence,
$$
\mathcal{G}(\mu_{c})=\max\{1-c,0\}\delta_{0}+\sigma\quad\text{with}\quad d\sigma=2t\,\mathbf{1}_{(\sqrt{\max\{c-1\,,0\}},\sqrt{c}\,]}(t)\,dt.
$$
\end{ej}

\vspace{0.3cm}
\section{Lyapunov exponents of free operators}

\vspace{0.2cm}
\noindent Let $\{a_{k}\}_{k=1}^{\infty}$ be free positive identically distributed operators. Let $\mu$ be the spectral probability measure of $a_{k}^{2}$ and assume that $\mu(\{0\})=0$. Using Theorem \ref{main} we know that 
the sequence of positive operators 
$$
\Big (a_{1}a_{2}\ldots a_{n}^2\ldots a_{2}a_{1}\Big)^{\frac{1}{2n}}
$$ 
converges in distribution to a positive operator $\Lambda$ with distribution $\nu$ in  $\mathcal{M}_{+}$. Since $\mu(\{0\})=0$, this distribution is absolutely continuous with respect to the Lebesgue measure and has Radon--Nikodym derivative
$$
d\nu(t)=f(t)\,\mathbf{1}_{(\norm{a_{1}^{-1}}_{2}^{-1},\norm{a_{1}}_{2}]}(t)\,dt
$$
where $f(t)=\big(F^{<-1>}_{\mu}\big)'(t)$ and $F_{\mu}(t)=S_{\mu}(t-1)^{-1/2}$.

\vspace{0.2cm}
\noindent Let $L$ be the selfadjoint, possibly unbounded operator, defined by $L:=\ln(\Lambda)$, and let $\gamma$ be the spectral probability distribution of $L$. It is a direct calculation to see that $\gamma$ is absolutely continuous with respect to Lebesgue measure and has Radon--Nikodym derivative 
$$
d\gamma(t)=e^{t}f(e^{t})\,\mathbf{1}_{(\ln\norm{a_{1}^{-1}}_2^{-1},\,\ln\norm{a_{1}}_2]}(t)\,dt.
$$

\vspace{0.2cm}
\noindent The probability distribution $\gamma$ of $L$ is what is called the distribution of the Lyapunov exponents (see \cite{N1}, \cite{N2} and \cite{Ru} and \cite{LyaV} for a more detailed exposition on Lyapunov exponents in the classical and non--classical case).

\vspace{0.3cm}
\begin{teo}
Let $\{a_{k}\}_{k=1}^{\infty}$ be free positive identically distributed invertible operators. Let $\mu$ be the spectral probability measure of $a_{k}^{2}$. Let $\gamma$ be probability distribution of the Lyapunov exponents associated to the sequence. Then 
$\gamma$ is absolutely continuous with respect to Lebesgue measure and has Radon--Nikodym derivative 
$$
d\gamma(t)=e^{t}f(e^{t})\,\mathbf{1}_{(\ln\norm{a_{1}^{-1}}_2^{-1},\,\ln\norm{a_{1}}_2]}(t)\,dt.
$$
where $f(t)=\big(F^{<-1>}_{\mu}\big)'(t)$ and $F_{\mu}(t)=S_{\mu}(t-1)^{-1/2}$.
\end{teo}

\vspace{0.3cm}
\begin{obs}
Note that if the operators $a_{k}$ are not invertibles in the $\norm{\cdot}_{2}$ then the selfadjoint operator $L$ is unbounded. See in the next example the case $\lambda=1$.
\end{obs}

\vspace{0.3cm}
\noindent The following is an example done previously in \cite{LyaV} using different techniques. 
\vspace{0.3cm}
\begin{ej}$\mathrm{(Marchenko-Pastur\,\, distribution)}$
Let $\{a_{k}\}_{k=1}^{\infty}$ be free positive identically distributed operators such that $a_{k}^{2}$ has the Marchenko--Pastur distribution $\mu$ of parameter $\lambda\geq 1$. Then as we saw in the Example \ref{ejmp}, in the last Section
$$
d\nu(t)=2t\,\mathbf{1}_{(\sqrt{\lambda-1},\sqrt{\lambda}\,]}(t)\,dt.
$$
Therefore, we see that the probability measure of the Lyapunov exponents is $\gamma$ with
$$
d\gamma(t)=2e^{2t}\,\mathbf{1}_{\big(\frac{1}{2}\ln(\lambda-1),\frac{1}{2}\ln(\lambda)\big]}(t)\,dt.
$$
If $\lambda=1$, this law is the exponential law discovered by C.M.Newman as a scaling limit of Lyapunov exponents of large random matrices. (See \cite{N1}, \cite{N2} and \cite{IsN}). This law is often called the ``triangle'' law since it implies that the exponentials of Lyapunov exponents converge to the law whose density is in the form of a triangle.
\end{ej}

\end{document}